\documentclass[11pt,a4paper]{amsart}
\usepackage{amssymb,amsmath,amsfonts}

\headsep=1.2500cm \numberwithin{equation}{section}
\newtheorem{Theorem}{Theorem}[section]
\newtheorem{Proposition}[Theorem]{Proposition}
\newtheorem{cor}[Theorem]{Corollary}

\theoremstyle{remark}

\begin{document}

\title{A note on approximately biflat Banach algebras}

\author[A. Sahami]{A. Sahami}

\address{Faculty of Mathematics and Computer Science,
Amirkabir University of Technology, 424 Hafez Avenue, 15914 Tehran,
Iran.}

\email{amir.sahami@aut.ac.ir}



\maketitle

\begin{abstract}
In this paper, we study the notion of approximately biflat Banach
algebras for second dual Banach algebras and semigroup algebras. We
show that for a locally compact group $G$, if $S(G)^{**}$ is
approximately biflat, then $G$ is amenable group. Also we give some
conditions which the second dual of a Triangular Banach algebra is
never approximately biflat. For a uniformly locally finite semigroup
$S$, we show that $\ell^{1}(S)$ is approximately biflat if and only
if $\ell^{1}(S)$ is biflat.
\end{abstract}
\section{Introduction}
Helemskii defined the notion of biflat Banach algebras. In fact a
Banach algebra $A$ is biflat if there exists a bounded $A$-bimodule
$\rho:(A\otimes_{p}A)^{*}\rightarrow A^{*}$ such that
$\rho\circ\pi^{*}_{A}(f)=f$ for each $f\in A^{*}$, where
$\pi_{A}:A\otimes_{p}A\rightarrow A$ is defined by $\pi_{A}(a\otimes
b)=ab$ for each $a,b\in A.$ For a group algebra $L^{1}(G)$
associated with a locally compact group $G$, $L^{1}(G)$ is biflat if
and only if $G$ is amenable group. For the further details see
\cite{hel}. Recently Ramsden in \cite{rams} characterized the
biflatness os semigroup algebras associated to a locally finite
inverse semigroup. He showed that for a  locally finite inverse
semigroup $S$, $\ell^{1}(S)$ is biflat if and only if each $G_{p}$
is amenable group, where $p\in E(S)$ and $G_{p}$ is a maximal
subgroup of $S$. Also biflatness of Triangular Banach algebras have
been studied in \cite{sat}.\\
Approximate notions in the homology of  Banach algebras introduced
 and have been under more observations, recently. A Banach algebra
 $A$ is approximately biflat if there exists a a net of $A$-bimodule
 morphisms $(\rho_{\alpha})$ from $(A\otimes_{p}A)^{*}$  into
 $A^{*}$ such that
 $\rho_{\alpha}\circ\pi^{*}_{A}\xrightarrow{W^{*}OT}id_{A^{*}},$ where
$W^{*}OT$ is denoted for the weak-star operator topology and
$id_{A^{*}}$ the identity map on $A^{*}.$ In fact for the discrete
Heisenberg group $G$ the Fourier algebra $A(G)$ is approximately
biflat but $A(G)$ is not biflat, see \cite{sam}. Samei {\it et al.}
also showed that if $A$ is an approximately biflat Banach algebra
with an approximate identity, then $A$ is pseudo-amenable, that is,
there exists a net $(m_{\alpha})$ in $A\otimes_{p}A$ such that
$$a\cdot m_{\alpha}-m_{\alpha}\cdot a,\quad
\pi_{A}(m_{\alpha})a\rightarrow a\quad (a\in A).$$ For further
details about pseudo-amenability, see \cite{ghah pse}.

Motivated by these considerations, we study approximate biflatness
of $\ell^{1}(S)$, where $S$ is a uniformly locally finite semigroup.
We show that $\ell^{1}(S)$ is approximately biflat if and only if
$\ell^{1}(S)$ is biflat. Also we show that  approximately biflatness
of  $\ell^{1}(S)^{**}$ implies the pseudo-amenability of
$\ell^{1}(S).$ Also for a locally compact group $G$, we show that
approximately biflatness of $S(G)^{**}$, implies that $G$ is
amenable, where $S(G)$ is a Segal algebra with respect to $G$.
Finally we give a criteria to study approximately biflatness of
Triangular Banach algebras. We show that some second dual of
Triangular Banach algebras related to a locally compact groups are
never approximately biflat.
\section{Preliminaries}
Let $A$ be a Banach algebra. We recall that if $X$ is a Banach
$A$-bimodule, then  $X^{*}$ is also a Banach $A$-bimodule via  the
following actions
$$(a\cdot f)(x)=f(x\cdot a) ,\hspace{.25cm}(f\cdot a)(x)=f(a\cdot x ) \hspace{.5cm}(a\in A,x\in X,f\in X^{*}). $$

Throughout, the character space of $A$ is denoted by $\Delta(A)$,
that is, all non-zero multiplicative linear functionals on $A$. Let
$\phi\in \Delta(A)$. Then $\phi$ has a unique extension
$\tilde{\phi}\in\Delta(A^{**})$
 which is defined by $\tilde{\phi}(F)=F(\phi)$ for every
$F\in A^{**}$.

Let $\{A_{\alpha}\}_{\alpha\in \Gamma}$ be a collection of Banach
algebras. Then we define the  $\ell^{1}$-direct sum of $A_{\alpha}$
by  $$\ell^{1}-\oplus_{\alpha \in \Gamma}
A_{\alpha}=\{(a_{\alpha})\in \prod_{\alpha\in
\Gamma}A_{\alpha}:\sum_{\alpha\in\Gamma}||a_{\alpha}||<\infty\}.$$
It is easy to verify that
$$\Delta(\ell^{1}-\oplus_{\alpha\in\Gamma} A_{\alpha})=\{\oplus\phi_{\beta}:\phi_{\beta}\in\Delta(A_{\beta}),\beta\in\Gamma\},$$
where $\oplus\phi_{\beta}((a_{\alpha})_{\alpha\in
\Gamma})=\phi_{\beta}(a_{\beta})$ for every $(a_{\alpha})_{\alpha\in
\Gamma}\in \ell^{1}-\oplus_{\alpha \in \Gamma} A_{\alpha}$ and every
$\beta\in\Gamma$.

Let $A$ be a Banach algebra and  let $\Lambda$ be a non-empty set.
The set of all $\Lambda\times\Lambda$ matrixes $(a_{i,j})_{i,j}$
which entries come from $A$ is denoted by $\mathbb{M}_{\Lambda}(A)$.
With the matrix multiplication and the following norm
$$||(a_{i,j})_{i,j}||=\sum_{i,j}||a_{i,j}||<\infty,$$
$\mathbb{M}_{\Lambda}(A)$ is a Banach algebra.
$\mathbb{M}_{\Lambda}(A)$ belongs to the class of $\ell^{1}$-Munn
algebras.  The map $\theta:\mathbb{M}_{\Lambda}(A)\rightarrow
A\otimes_{p} \mathbb{M}_{\Lambda}(\mathbb{C})$ defined by
$\theta((a_{i,j}))=\sum_{i,j}a_{i,j}\otimes E_{i,j}$ is an isometric
algebra isomorphism, where $(E_{i,j})$ denotes the matrix unit of
$\mathbb{M}_{\Lambda}(\mathbb{C})$.
 Also it is well-known that $
\mathbb{M}_{\Lambda}(\mathbb{C})$ is a  biprojective Banach algebra
\cite[Proposition 2.7]{rams}.

The main  reference for the semigroup theory is \cite{how}. Let $S$
be a semigroup and let $E(S)$ be the set of its idempotents. A
partial order on $E(S)$ is defined by
$$s\leq t\Longleftrightarrow s=st=ts\quad (s,t\in E(S)).$$
If $S$ is an inverse semigroup, then there exists a partial order on
$S$ which is coincide with the partial order on $E(S)$. Indeed
$$s\leq t\Longleftrightarrow s=ss^{*}t\quad (s,t\in
S).$$ For every  $x\in S$, we denote $(x]=\{y\in S|\,y\leq x\}$. $S$
is called locally finite (uniformly locally finite) if for each
$x\in S$, $|(x]|<\infty\,\,(\sup\{|(x]|\,:\,x\in S\}<\infty)$,
respectively.

Suppose that $S$ is an inverse semigroup. Then  the maximal subgroup
of $S$ at $p\in E(S)$ is denoted by $G_{p}=\{s\in
S|ss^{*}=s^{*}s=p\}$. For an inverse semigroup $S$ there exists a
relation $\mathfrak{D}$ such that $s\mathfrak{D}t$ if and only if
there exists $x\in S$ such that $ss^{*}=xx^{*}$ and $t^{*}t=x^{*}x$.
We denote $\{\mathfrak{D}_{\lambda}:\lambda\in \Lambda\}$ for the
collection of $\mathfrak{D}$-classes and
$E(\mathfrak{D}_{\lambda})=E(S)\cap \mathfrak{D}_{\lambda}.$ An
inverse semigroup $S$ is called Clifford if for each $s\in S$, there
exists $s^{*}\in S$ such that $ss^{*}=s^{*}s.$
\section{Approximate biflatness of second  dual of Banach algebras}
In this section we investigate approximate biflatness dual Banach
algebras.
\begin{Proposition}\label{net}
Let $A$ be a Banach algebra. $A$ is approximately biflat Banach
algebra if and only if there exists a net $(\rho_{\alpha})$ of
bounded $A$-bimodule morphisms from $A$ into $(A\otimes_{p}A)^{**}$
such that $\pi^{**}_{A}\circ\rho_{\alpha}(a)\rightarrow a$ for every
$a\in A.$
\end{Proposition}
\begin{proof}
Let $A$ be approximately biflat. Then there exists a net
$\xi_{\alpha}:A^{*}\rightarrow (A\otimes_{p}A)^{*}$ such that
$\xi_{\alpha}\circ\pi^{*}_{A}(f)-f\rightarrow 0,$ for every $f\in
A^{*}$. Set $\rho_{\alpha}=\xi^{*}_{\alpha}$, hence for each $a\in
A$ and $f\in A^{*}$ with $||f||\leq 1$, we have
\begin{equation*}
\begin{split}
||\pi^{**}_{A}\circ\rho_{\alpha}(a)-a||=||\pi^{**}_{A}\circ\xi^{*}_{\alpha}(a)-a||&=||(\xi_{\alpha}\circ\pi^{*}_{A})^{*}(a)-a||\\
&=||(\xi_{\alpha}\circ\pi^{*}_{A})^{*}(a)(f) -a(f)||\\
&=||a(\xi_{\alpha}\circ\pi^{*}_{A}(f))-a(f)||\\
&=||a(\xi_{\alpha}\circ\pi^{*}_{A}(f)-f)||\\
&\leq ||a||||\xi_{\alpha}\circ\pi^{*}_{A}(f)-f||\rightarrow 0.
\end{split}
\end{equation*}
For converse, suppose that there exists a net $(\rho_{\alpha})$ of
bounded $A$-bimodule morphisms from $A$ into $(A\otimes_{p}A)^{**}$
such that $\pi^{**}_{A}\circ\rho_{\alpha}(a)\rightarrow a$ for every
$a\in A.$ Set
$\xi_{\alpha}=\rho_{\alpha}^{*}|_{(A\otimes_{p}A)^{*}}$.
\begin{equation*}
\begin{split}
\xi_{\alpha}\circ\pi^{*}_{A}(f)(a)-f(a)=f(\pi^{**}_{A}\circ\rho_{\alpha}(a)-a)\rightarrow
0,
\end{split}
\end{equation*}
where $f\in A^{*}, a\in A.$
\end{proof}

\begin{Theorem}\label{dual}
Let $A$ be a Banach algebra with an approximate identity. If
$A^{**}$ is approximately biflat, then $A$ is pseudo-amenable.
\end{Theorem}
\begin{proof}
Suppose that $A^{**}$ is approximately biflat. Then by Proposition
\ref{net} there exists a net $(\rho_{\alpha})_{\alpha\in I}$ of
$A^{**}$-bimodule morphism from $A^{**}$ into
$(A^{**}\otimes_{p}A^{**})^{**}$ such that
$\pi_{A^{**}}^{**}\circ\rho_{\alpha}(a)\rightarrow a$ for each $a\in
A^{**}.$ It is easy to see that the net $(\rho_{\alpha}|_{A})$
 is also a net of $A$-bimodule morphism satisfies $\pi_{A^{**}}^{**}\circ\rho_{\alpha}|_{A}(a)\rightarrow a$ for each
$a\in A.$ We denote $(e_{\lambda})_{\lambda\in J}$ for the
approximate identity of $A.$ Consider
$$\lim_{\alpha}\lim_{\lambda}a\cdot \rho_{\alpha}(e_{\lambda})-\rho_{\alpha}(e_{\lambda})\cdot a=\lim_{\alpha}\lim_{\lambda}\rho_{\alpha}(ae_{\lambda}-e_{\lambda}a)
=\lim_{\alpha}\rho_{\alpha}(0)=0\quad (a\in A).$$ Also
$$\lim_{\alpha}\lim_{\lambda}\pi_{A^{**}}^{**}\circ\rho_{\alpha}(e_{\lambda})a-a=\lim_{\alpha}\pi_{A^{**}}^{**}\circ\rho_{\alpha}(a)-a=0,\quad (a\in A).$$
Let $E=I\times J^{I}$ be a directed set with product ordering, that
is $$(\alpha,\beta)\leq_{E} (\alpha^{'},\beta^{'})\Leftrightarrow
\alpha\leq_{I} \alpha^{'}, \beta\leq_{J^{I}}\beta^{'}\qquad
(\alpha,\alpha^{'}\in I,\quad \beta,\beta^{'}\in J^{I}),$$ where
$J^{I}$ is the set of all functions from $I$ into $J$ and $
\beta\leq_{J^{I}}\beta^{'}$ means that $\beta(d)\leq_{J}
\beta^{'}(d)$ for each $d\in I$. Suppose that
$\gamma=(\alpha,\beta_{\alpha})\in E$ and
$m_{\gamma}=\rho_{\alpha}(e_{\lambda_{\alpha}})\in
(A^{**}\otimes_{p}A^{**})^{**}.$ Applying iterated limit theorem
\cite[page 69]{kel} and above calculations, we can easily see that
$$a\cdot m_{\gamma}-m_{\gamma}\cdot a\rightarrow 0,\quad
\pi_{A^{**}}^{**}(m_{\gamma})a\rightarrow a,\quad(a\in A).$$
 There exists a
bounded linear map $\psi:A^{**}\otimes_{p} A^{**}\rightarrow
(A\otimes_{p} A)^{**}$ such that for $a,b\in A$ and $m\in
A^{**}\otimes_{p} A^{**}$, the following holds;
\begin{enumerate}
\item [(i)] $\psi(a\otimes b)=a\otimes b $,
\item [(ii)] $\psi(m)\cdot a=\psi(m\cdot a)$,\qquad
$a\cdot\psi(m)=\psi(a\cdot m),$
\item [(iii)] $\pi_{A}^{**}(\psi(m))=\pi_{A^{**}}(m),$
\end{enumerate}
see \cite[Lemma 1.7]{gha loy}. So $\psi^{**}(m_{\gamma})$ is a net
in $(A\otimes_{p}A)^{****}$ such that
$$a\cdot\psi^{**}(m_{\gamma})-\psi^{**}(m_{\gamma})\cdot a\rightarrow 0,\quad \pi_{A}^{****}(\psi^{**}(m_{\gamma}))a
=\pi^{**}_{A^{**}}(m_{\gamma})a\rightarrow a\qquad(a\in A).$$ Put
$n_{\gamma}=\psi^{**}(m_{\gamma}).$ Suppose that $\epsilon>0$ and
$F=\{a_{1},...,a_{r}\}\subseteq A$. Set
\begin{equation*}
\begin{split}
V=&\{(a_{1} \cdot n-n\cdot a_{1},..., a_{r} \cdot n-n\cdot a_{r},\pi_{A}^{**}(n)a_{1}-a_{1},...,\pi_{A}^{**}(n)a_{r}-a_{r})|n\in (A\otimes_{p}A)^{**}\}\\
&\subseteq (\prod^{r}_{i=1}(A\otimes_{p}A)^{**})\oplus_{1}
(\prod^{r}_{i=1}A^{**}).
\end{split}
\end{equation*}
It is easy to see that $(0,0,...,0)$ is a $w$-limit point of $V$.
Since $V$ is convex set $\overline{V}^{||\cdot||}=\overline{V}^{w}$,
then $(0,0,...,0)$ is a $||\cdot||$-limit point of $V$. Hence there
exists a net $(n_{(F,\epsilon)})$ in $(A\otimes_{p}A)^{**}$ such
that
$$||a_{i}\cdot n_{(F,\epsilon)}-n_{(F,\epsilon)}\cdot a_{i}||<\epsilon,\quad||\pi_{A}^{**}(n_{(F,\epsilon)})a_{i}-a_{i}||<\epsilon,\quad (i\in\{1,2,..., r\}).$$
 Observe that
$$\Delta=\{(F,\epsilon):F {\hbox{ is a finite subset of }}A, \epsilon>0\},$$ with the following order
$$(F,\epsilon)\leq (F^{\prime},\epsilon^{\prime})\Longrightarrow F\subseteq F^{\prime},\quad \epsilon\geq \epsilon^{\prime}$$
is a directed set. It is easy to see that there exists a net
$(n_{(F,\epsilon)})_{(F,\epsilon)\in \Delta}$  in
$(A\otimes_{p}A)^{**}$ such that
$$a \cdot n_{(F,\epsilon)}-n_{(F,\epsilon)}\cdot a \rightarrow 0,\quad \pi_{A}^{**}(n_{(F,\epsilon)})a-a\rightarrow 0,$$
for every $a\in A$. Using the same method as above we can assume
that $(n_{(F,\epsilon)})_{(F,\epsilon)\in \Delta}$ is a subset of
$A\otimes_{p}A$. This means that $A$ is pseudo-amenable.
\end{proof}
Let $A$ be a Banach algebra and $\phi\in\Delta(A)$. We say that $A$
is approximately $\phi$-inner amenable if there exists a net
$(a_{\alpha})_{\alpha}$ in $A$ such that
$aa_{\alpha}-aa_{\alpha}a\rightarrow 0$ and
$\phi(a_{\alpha})\rightarrow 1,$ for all $a\in A.$ Also $A$ is
approximately left $\phi-$amenable if there exists a net
$m_{\alpha}$ in $A$ such that
$am_{\alpha}-\phi(a)m_{\alpha}\rightarrow 0$ and
$\phi(m_{\alpha})\rightarrow 1$, see \cite{agha}.

The proof of following two results are similar to the proof of
Theorem \ref{dual} which we omit it.
\begin{Theorem}
Suppose that $A$ is an approximately $\phi$-inner amenable Banach
algebra. If $A^{**}$ is approximately biflat, then $A$ is
approximately left $\phi$-amenable.
\end{Theorem}
\begin{cor}
Suppose that $A$ is an approximately $\phi$-inner amenable Banach
algebra. If $A$ is approximately biflat, then $A$ is approximately
left $\phi$-amenable.
\end{cor}
\begin{Theorem}\label{for tri}
Let $A$ be a Banach algebra with $\phi\in\Delta(A)$. Suppose that
$\overline{A ker\phi}=ker\phi$. If $A^{**}$ is approximately biflat,
then $A$ is left $\phi$-amenable.
\end{Theorem}
\begin{proof}
Since $A^{**}$ is approximately biflat, there exists a net of
$A^{**}$-bimodule morphisms $(\rho_{\alpha})$ from $A^{**}$  into
$(A^{**}\otimes_{p}A^{**})^{**}$ such that
$\pi^{**}_{A^{**}}\circ\rho_{\alpha}(a)\rightarrow a\quad(a\in
A^{**})$. We denote $id:A\rightarrow A$ for the identity map and
$q:A\rightarrow \frac{A}{L}$ the quotient map, where $L=ker\phi.$
Also it is well-known that  there exists a bounded linear map
$\psi:A^{**}\otimes_{p} A^{**}\rightarrow (A\otimes_{p} A)^{**}$
such that for $a,b\in A$ and $m\in A^{**}\otimes_{p} A^{**}$, the
following holds;
\begin{enumerate}
\item [(i)] $\psi(a\otimes b)=a\otimes b $,
\item [(ii)] $\psi(m)\cdot a=\psi(m\cdot a)$,\qquad
$a\cdot\psi(m)=\psi(a\cdot m),$
\item [(iii)] $\pi_{A}^{**}(\psi(m))=\pi_{A^{**}}(m),$
\end{enumerate}
see \cite[Lemma 1.7]{gha loy}. Set $\eta_{\alpha}:=(id\otimes
q)^{****}\circ\psi^{**}\circ\rho_{\alpha}|_{A}:A\rightarrow
(A\otimes_{p}\frac{A}{L})^{****}$ for each $\alpha.$ We claim that
$\eta_{\alpha}(l)=0$ for each $l\in \ker\phi.$ To see this let
$l\in\ker\phi$ be an arbitrary element. Since $\overline{AL}=L$,
there exist two nets $(a_{\beta})$ in $A$ and $(l_{\beta})$ in $L $
such that $l=\lim_{\beta}a_{\beta}l_{\beta}$. Consider
\begin{equation}
\begin{split}
\eta_{\alpha}(l)=(id\otimes
q)^{****}\circ\psi^{**}\circ\rho_{\alpha}(l)&=(id\otimes
q)^{****}\circ\psi^{**}\circ\rho_{\alpha}(l)\\
&=\lim_{\beta}(id\otimes
q)^{****}\circ\psi^{**}\circ\rho_{\alpha}(a_{\beta}l_{\beta})\\
&=\lim_{\beta}(id\otimes
q)^{****}(\psi^{**}\circ\rho_{\alpha}(a_{\beta})\cdot l_{\beta})=0.
\end{split}
\end{equation}
Hence, for each $\alpha,$ $\eta_{\alpha}$ can induce a map on
$\frac{A}{L}$ which we again denote it by $\eta_{\alpha}$. We also
denote $\overline{\phi}$ for a character which induced by $\phi$ on
$\frac{A}{L}$ given by $$\overline{\phi}(a+L)=\phi(a)\quad (a\in
A).$$ Set $$g_{\alpha}=(id\otimes \overline{\phi})^{****}\circ
\eta_{\alpha}:\frac{A}{L}\rightarrow A^{**}.$$ Pick an element
$x_{0}\in A$ such that $\phi(x_{0})=1.$ Define
$m_{\alpha}=g_{\alpha}(x_{0}+L).$ We know that
 $(g_{\alpha})$ is a net of left $A$-module morphisms.  Thus
\begin{equation}
\begin{split}
am_{\alpha}=a(id\otimes \overline{\phi})^{****}\circ
\eta_{\alpha}(x_{0}+L)&=(id\otimes \overline{\phi})^{****}\circ
\eta_{\alpha}(ax_{0}+L)\\
&=\phi(a)(id\otimes \overline{\phi})^{****}\circ
\eta_{\alpha}(x_{0}+L)\\
&=\phi(a)m_{\alpha},
\end{split}
\end{equation}
the last equality holds because $ax_{0}-\phi(a)x_{0}\in L.$ Since
$$\widetilde{\widetilde{\phi}}\circ (id\otimes
\overline{\phi})^{****}=(\phi\otimes \overline{\phi})^{****},\quad
(\phi\otimes \overline{\phi})^{****}\circ(id\otimes q)^{****}
=\widetilde{\widetilde{\phi}}\circ\pi^{****}_{A},$$ we have
\begin{equation}
\begin{split}
\widetilde{\widetilde{\phi}}(m_{\alpha})=\widetilde{\widetilde{\phi}}\circ(id\otimes
\overline{\phi})^{****}\circ
\eta_{\alpha}(x_{0}+L)&=\widetilde{\widetilde{\phi}}\circ(id\otimes
\overline{\phi})^{****}\circ
\eta_{\alpha}(x_{0})\\
&=\widetilde{\widetilde{\phi}}\circ(id\otimes
\overline{\phi})^{****}\circ(id\otimes q)^{****}\circ
\rho_{\alpha}(x_{0})\\
&=\widetilde{\widetilde{\phi}}\circ\pi^{****}_{A}\circ\rho_{\alpha}(x_{0})\\
&\rightarrow \phi(x_{0})=1.
\end{split}
\end{equation}
Replacing $(m_{\alpha})$ with
$(\frac{m_{\alpha}}{\phi(m_{\alpha})})$ on can find an element $m\in
A^{****}$ such that $am=\phi(a)m$ and
$\widetilde{\widetilde{\phi}}(m)=1$ for every $a\in A$. Let
$F=\{a_{1},a_{2},...,a_{r}\}$ be an arbitrary finite subset of $A$
and $\epsilon>0$. Set
$$V=\{(a_{1}n-\phi(a_{1})n,a_{2}n-\phi(a_{2})n,..., a_{r}n-\phi(a_{r})n, \widetilde{\phi}(n)-1)|n\in A^{**},||n||\leq ||m||\}.$$
It is easy to see that $V$ is a convex subset of
$\prod^{r}_{i=1}A^{**}\oplus_{1}\mathbb{C}$. It is easy to see that
$(0,0,...,0)\in \overline{V}^{w}=\overline{V}^{||\cdot||}$. Thus
there exists a bounded net $(n_{(F,\epsilon)})_{(F,\epsilon)}$ in
$A^{**}$ such that
$$||a_{i}n_{(F,\epsilon)}-\phi(a_{i})n_{(F,\epsilon)}||<\epsilon,\quad |\tilde{\phi}(n_{(F,\epsilon)})-1|<\epsilon,\quad a_{i}\in F.$$
It is easy to see that
$$\Delta=\{(F,\epsilon):F {\hbox{ is a finite subset of }}A, \epsilon>0\},$$ with the following order
$$(F,\epsilon)\leq (F^{\prime},\epsilon^{\prime})\Longrightarrow F\subseteq F^{\prime},\quad \epsilon\geq \epsilon^{\prime}$$
is a directed set. Therefore there exists a bounded net
$(n_{(F,\epsilon)})_{(F,\epsilon)\in \Delta}$  in $A^{**}$ such that
$$an_{(F,\epsilon)}-\phi(a)n_{(F,\epsilon)}\rightarrow 0,\quad \tilde{\phi}(n_{(F,\epsilon)})-1\rightarrow 0,\quad a\in A.$$
Since $(n_{(F,\epsilon)})_{(F,\epsilon)}$ is a bounded net in
$A^{**}$, then $(n_{(F,\epsilon)})_{(F,\epsilon)}$ has a
$w^{*}$-limit point  in $A^{**}$, say $N$. It is easy to see that
$$aN=\phi(a)N,\quad \tilde{\phi}(N)=1\quad (a\in A).$$
It means that $A$ is left $\phi$-amenable.
\end{proof}
The map $\phi_{1}:L^{1}(G)\rightarrow \mathbb{C}$ which specified by
$$\phi_{1}(f)=\int_{G} f(x)dx$$ is called augmentation character. We
know that augmentation character induce a character on $S(G)$ which
we denote it by $\phi_{1}$ again, see \cite{alagh}.

We recall that, for a locally compact group $G$, a linear subspace
$S(G)$ of  $L^{1}(G)$ is said to be a Segal algebra on $G$ if it
satisfies the following properties:
\begin{enumerate}
\item [(i)] $S(G)$ is a dense left ideal  in $L^{1}(G)$;
\item [(ii)]  $S(G)$ with respect to some norm $||\cdot||_{S(G)}$ is
a Banach space and $|| f||_{L^{1}(G)}\leq|| f||_{S(G)}$;
\item [(iii)] For $f\in S(G)$ and $y\in G$, $L_{y}f\in S(G)$
and the map $y\mapsto \delta_{y}\ast f$ is continuous. Also
$||\delta_{y}\ast f||_{S(G)}=|| f||_{S(G)}$, for $f\in S(G)$ and
$y\in G$.
\end{enumerate}
For more information about this algebras see \cite{rei}.
\begin{cor}
Let $G$ be a locally compact group. If $S(G)^{**}$ is approximately
biflat, then $G$ is amenable.
\end{cor}
\begin{proof}
It is well-known that every Segal algebra has a left approximate
identity. Suppose that $\phi\in\Delta(S(G))$. It is easy to see that
$\overline{S(G)\ker\phi}=\ker\phi$. Using the  Theorem \ref{for
tri}, approximate biflatness of $S(G)^{**}$ implies that $S(G)$ is
left $\phi$-amenable. Now by \cite[Corollary 3.4]{alagh} $G$ is
amenable.
\end{proof}
\section{Approximate biflatness of certain semigroup algebras}
In this section we study approximate biflatness of some semigroup
algebras.\\
Before giving the following proposition we have to give some
backgrounds. Suppose that $A$ and $B$ are Banach algebras and also
suppose that $E$ and $F$ are Banach $A$-bimodule and Banach
$B$-bimodule, respectively. Via the following module actions, one
can see that $E\otimes_{p}F$ becomes a Banach
$A\otimes_{p}B$-bimodule:
$$(a\otimes b)\cdot(x\otimes y)=(a\cdot x)\otimes(b\cdot y),\quad (x\otimes y)\cdot(a\otimes b)=(x\cdot a)\otimes(y\cdot b),$$
for each $a\in A, x\in E,b\in B, y\in F.$ One can readily see that
$B(E,F)$ (the set of all bounded linear operator from $E$ into $F$)
is a Banach $A\otimes_{p}B$-bimodule via the following actions:
$$((a\otimes b)\ast T)(x)=b\cdot T(x\cdot a),\quad (T\ast (a\otimes b))(x)=T(a\cdot x)\cdot b,$$
for each $T\in B(E,F), a\in A,b\in B,x\in E.$ We denote this Banach
$A\otimes_{p}B$-bimodule by $\tilde{B}(E,F)$. Also we can see that
$B(F,E)$ becomes a  Banach $A\otimes_{p}B$-bimodule via the
following actions:
$$((a\otimes b)\ast T)(x)=a\cdot T(x\cdot b),T\ast (a\otimes b))(x)=T(b\cdot x)\cdot a, $$
for each $T\in B(E,F), a\in A,b\in B,x\in E.$ We denote this Banach
$A\otimes_{p}B$-bimodule by $\widehat{B}(E,F)$. Note that for each
$\lambda\in (E\otimes_{p} F)^{*}$ we can define
$\widetilde{T_{\lambda}}\in B(E,F^{*})$ and
$\widehat{T_{\lambda}}\in B(F,E^{*})$ by
$$<y,\widetilde{T_{\lambda}}(x)>=<x\otimes y,\lambda>,\quad <x,\widehat{T_{\lambda}}(y)>=<x\otimes y,\lambda>,$$
for each $x\in E,y\in F.$ The map
$\widetilde{\xi}:(E\otimes_{p}F)^{*}\rightarrow \tilde{B}(E,F^{*})$
given by $\widetilde{\xi}(\lambda)=\widetilde{T}_{\lambda}$ is an
isometric $A\otimes_{p}B$-bimodule isomorphism. Also the map
$\widehat{\xi}:(E\otimes_{p}F)^{*}\rightarrow \widehat{B}(F,E^{*})$
given by $\widehat{\xi}(\lambda)=\widehat{T_{\lambda}}$ is an
isometric $A\otimes_{p}B$-bimodule isomorphism. Then there exists a
bounded isometric $A\otimes_{p}B$-bimodule isomorphism from
$\widetilde{B}(E,F^{*})$ into $\widehat{B}(F,E^{*})$ which denoted
by $L$. Also we have to remind that there exists  an isometric
$A\otimes_{p}B$-bimodule isomorphism from
$(A\otimes_{p}A)\otimes_{p}(B\otimes_{p}B)$ into
$(A\otimes_{p}B)\otimes_{p}(A\otimes_{p}B)$ defined by
$$\theta(a\otimes a^{\prime}\otimes b\otimes b^{\prime})=a\otimes b\otimes a^{\prime}\otimes b^{\prime}\quad (a,a^{\prime}\in A,b,b^{\prime}\in
B).$$ It is easy to see that $\theta$ is a bounded
$A\otimes_{p}B$-bimodule morphism.
\begin{Proposition}\label{tensor without unit}
Let $A$ be a biflat Banach algebra and $B$ be approximate biflat
Banach algebra. Then $A\otimes_{p}B$ is approximate biflat.
\end{Proposition}
\begin{proof}
Since $A$ is approximate biflat, there exists a net of bounded
$A$-bimodules from $A$ into $(A\otimes_{p}A)^{**}$, say
$(\rho_{\alpha})_{\alpha\in I}$, such that
$\rho_{\alpha}\circ\pi^{*}_{A}(f)-f\rightarrow 0,$ for each $f\in
A^{*}.$ Also since $B$ is biflat, there exists a  bounded
$B$-bimodules from $B$ into $(B\otimes_{p}B)^{**}$, say $\rho$, such
that $\rho\circ\pi^{*}_{B}(g)=g$ for each $g\in B^{*}.$
 Set
\begin{equation}
\begin{split}
\bar{\rho}_{\alpha}:((A\otimes_{p}B)\otimes_{p}(A\otimes_{p}B))^{*}&\xrightarrow{\theta^{*}}((A\otimes_{p}A)\otimes_{p}(B\otimes_{p}B))^{*}\\
&\xrightarrow{\widetilde{\xi}}
\widetilde{B}(A\otimes_{p}A,(B\otimes_{p}B)^{*})\\
&\xrightarrow{T\mapsto\rho\circ
T}\widetilde{B}(A\otimes_{p}A,B^{*})\\
&\xrightarrow{L}\widehat{B}(B,(A\otimes_{p}A)^{*})\\
&\xrightarrow{T\mapsto\rho_{\alpha}\circ
T}\widehat{B}(B,A^{*})\\
&\xrightarrow{\widehat{\xi^{-1}}}(A\otimes_{p}B)^{*}.
\end{split}
\end{equation}
Since $\bar{\rho}_{\alpha}$ is a composition of some
$(A\otimes_{p}B)$-bimodule morphisms, then $\bar{\rho}_{\alpha}$ is
a net of $(A\otimes_{p}B)$-bimodule morphisms. Take $\lambda\in
(A\otimes_{p}B)^{*}$. Using the following facts
$$\rho\circ\pi^{*}_{B}\circ\widetilde{T}_{\lambda}\circ\pi_{A}=\widetilde{T}_{\lambda}\circ\pi_{A}$$
and
$$L\circ \widetilde{T}_{\lambda}\circ\pi_{A}=\pi^{*}_{A}\circ\widehat{T}_{\lambda}$$
we have
$$\overline{\rho}_{\alpha}\circ \pi^{*}_{A\otimes_{p}B}(\lambda)-\lambda=\widehat{\xi^{-1}}\circ\rho_{\alpha}\circ\pi^{*}_{A}\circ\widehat{T_{\lambda}}-\lambda\rightarrow \lambda-\lambda=0.$$
This finishes the proof.
\end{proof}
We partially shows the converse of above Proposition in the
following theorem.
\begin{Theorem}\label{wit unit}
Let $A$ and $B$ be Banach algebras. Suppose that $A$ has an identity
and $B$ has a non-zero idempotent. If $A\otimes_{p}B$ is
approximately biflat, then $A$ is approximately biflat, so $A$ is
pseudo-amenable.
\end{Theorem}
\begin{proof}
Suppose that $A\otimes_{p}B$ is approximately biflat. Then by
Proposition \ref{net} there exists a net $(\rho_{\alpha})$ of
$A$-bimodule morphisms from $A\otimes_{p}B$ into
$((A\otimes_{p}B)\otimes_{p}(A\otimes_{p}B))^{**} $ such that
$\pi_{A\otimes_{p}B}^{**}\circ\rho_{\alpha}(x)\rightarrow x$ for
each $x\in A\otimes_{p}B.$ Take $e\in A$ the identity and $b_{0}\in
B$ the non-zero identity. It is easy to see that $A\otimes_{p}B$
becomes a Banach $A$-bimodule via the following actions:
$$a_{1}\cdot (a_{2}\otimes b)=a_{1}a_{2}\otimes b,\quad (a_{2}\otimes b)\cdot a_{1}=a_{2}a_{1}\otimes b\qquad (a_{1},a_{2}\in A, b\in B).$$
For each $\alpha$, we have
\begin{equation}
\begin{split}
\rho_{\alpha}(a_{1}a_{2}\otimes b_{0})=\rho_{\alpha}((a_{1}\otimes b_{0})(a_{2}\otimes b_{0}))&=(a_{1}\otimes b_{0})\cdot\rho_{\alpha}(a_{2}\otimes b_{0})\\
&=(a_{1}\cdot (e\otimes b_{0}))\cdot\rho_{\alpha}(a_{2}\otimes b_{0})\\
&=a_{1}\cdot \rho_{\alpha}(ea_{2}\otimes b_{0}b_{0})\\
&=a_{1}\cdot \rho_{\alpha}(a_{2}\otimes b_{0}).
\end{split}
\end{equation}
For each $\alpha$, we can also see that
$$\rho_{\alpha}((a_{2}\otimes b_{0})\cdot a_{1})=\rho_{\alpha}(a_{2}\otimes b_{0})\cdot a_{1}.$$
For each $\alpha$ set
$\overline{\rho}_{\alpha}(a)=\rho_{\alpha}(a\otimes b_{0})$. It is
easy to see that $(\overline{\rho}_{\alpha})$ is a net of
$A$-bimodule morphisms. Since $b_{0}$ is a non-zero element in $B$,
by Hahn-Banach theorem there exists a functional $f\in B^{*}$ such
that $f(b_{0})=1.$ Define
$T:(A\otimes_{p}B)\otimes_{p}(A\otimes_{p}B)\rightarrow
A\otimes_{p}A$ by $$T(a\otimes b\otimes c \otimes d)=f(bd)a\otimes
c$$ for each $a,c \in A$ and $ b,d\in B.$ Clearly $T$ is a bounded
linear map. One can see that $\pi^{**}_{A}\circ
T^{**}=(id_{A}\otimes f)^{**}\circ\pi^{**}_{A\otimes_{p}B}$, where
$$id_{A}\otimes f(a\otimes b)=f(b)a\quad (a\in A, b\in B).$$
Set $\widetilde{\rho}_{\alpha}=T^{**}\circ
\overline{\rho}_{\alpha}$. It is easy to see that
$(\widetilde{\rho}_{\alpha})_{\alpha}$ is a net of bounded
$A$-bimodule morphisms. Since $f(b_{0})=1$, we have
\begin{equation}
\begin{split}
\pi^{**}_{A}\circ \widetilde{\rho}_{\alpha}(a)=\pi^{**}_{A}\circ T^{**}\circ\overline{\rho}_{\alpha}(a)&=(\pi_{A}\circ T)^{**}\circ \rho_{\alpha}(a\otimes b_{0})\\
&=(id_{A}\otimes f)^{**}\circ\pi^{**}_{A\otimes_{p}B}\cdot\rho_{\alpha}(a\otimes b_{0})\\
&\rightarrow (id_{A}\otimes f)^{**}(a\otimes b_{0})=a,
\end{split}
\end{equation}
for each $a\in A.$ Using Proposition \ref{net} $A$ is approximate
biflat. Since $A$ has an identity by \cite[Theorem 2.4]{sam}, $A$ is
pseudo-amenable.
\end{proof}
\begin{cor}
Let $A$ be a  Banach algebra. If $A$  is approximately biflat, then
$\mathbb{M}_{\Lambda}(A)$ is approximately biflat. Converse is true,
provided that $A$ has a unit.
\end{cor}
\begin{proof}
Let $A$ be approximately biflat. It is well-known that there exists
an isometrical isomorphism between $\mathbb{M}_{\Lambda}(A)$ and
$A\otimes_{p}\mathbb{M}_{\Lambda}(\mathbb{C})$. Using the fact that
$\mathbb{M}_{\Lambda}(\mathbb{C})$ is always biprojective, see
\cite[Proposition 2.7]{rams}, then
$\mathbb{M}_{\Lambda}(\mathbb{C})$ is biflat. Now by Proposition
\ref{tensor without unit},
$A\otimes_{p}\mathbb{M}_{\Lambda}(\mathbb{C})$ is approximately
biflat.

For converse, since $A$ is unital and
$\mathbb{M}_{\Lambda}(\mathbb{C})$ has a non-zero idempotent, then
by Theorem \ref{wit unit} the approximate biflatness of
$\mathbb{M}_{\Lambda}(A)\cong
A\otimes_{p}\mathbb{M}_{\Lambda}(\mathbb{C})$ implies approximate
biflatness of $A.$
\end{proof}
\begin{Theorem}
Let $S$ be an inverse semigroup such that $E(S)$ is uniformly
locally finite. $\ell^{1}(S)$ is approximately biflat if and only if
$\ell^{1}(S)$ is biflat.
\end{Theorem}
\begin{proof}
Suppose that $\ell^{1}(S)$ is approximately biflat. Then by
Proposition  \ref{net} there exists a bet of $\ell^{1}(S)$-bimodule
morphism from $ \ell^{1}(S)$ into $( \ell^{1}(S)\otimes_{p}
\ell^{1}(S))^{**}$ such that $\pi^{**}_{
\ell^{1}(S)}\circ\rho_{\alpha}(a) -a\rightarrow 0$ for each $a\in
\ell^{1}(S)$. Since $S$ is uniformly locally finite, by
\cite[Theorem 2.18]{rams} we have
 $$\ell^{1}(S)\cong
\ell^{1}-\bigoplus\{\mathbb{M}_{E(\mathfrak{D}_{\lambda})}(\ell^{1}(G_{p_{\lambda}}))\},$$
where is a $\mathfrak{D}$-class and $G_{p_{\lambda}}$ is a maximal
subgroup at $p_{\lambda}$. Then the map
$P_{p_{\lambda}}:\ell^{1}(S)\rightarrow
\mathbb{M}_{E(\mathfrak{D}_{\lambda})}(\ell^{1}(G_{p_{\lambda}}))$
is a continuous homomorphism  with a dense range. Define
$$\eta_{\alpha}:\mathbb{M}_{E(\mathfrak{D}_{\lambda})}(\ell^{1}(G_{p_{\lambda}}))\rightarrow(
\mathbb{M}_{E(\mathfrak{D}_{\lambda})}(\ell^{1}(G_{p_{\lambda}}))\otimes_{p}\mathbb{M}_{E(\mathfrak{D}_{\lambda})}(\ell^{1}(G_{p_{\lambda}})))^{**}$$
by $\eta_{\alpha}=(P\otimes
P)^{**}\circ\rho_{\alpha}|_{\mathbb{M}_{E(\mathfrak{D}_{\lambda})}(\ell^{1}(G_{p_{\lambda}}))}$.
It is easy to see that $(\eta_{\alpha})$ is a net of
$\mathbb{M}_{E(\mathfrak{D}_{\lambda})}(\ell^{1}(G_{p_{\lambda}}))$-bimodule
morphisms. For each $a\in
\mathbb{M}_{E(\mathfrak{D}_{\lambda})}(\ell^{1}(G_{p_{\lambda}}))$
we have
\begin{equation}\label{equ}
\begin{split}
\pi^{**}_{\mathbb{M}_{E(\mathfrak{D}_{\lambda})}(\ell^{1}(G_{p_{\lambda}}))}\circ\eta_{\alpha}(a)&=\pi^{**}_{\mathbb{M}_{E(\mathfrak{D}_{\lambda})}(\ell^{1}(G_{p_{\lambda}}))}
\circ (P_{p_{\lambda}}\otimes P_{p_{\lambda}})^{**}\circ\rho_{\alpha}|_{\mathbb{M}_{E(\mathfrak{D}_{\lambda})}(\ell^{1}(G_{p_{\lambda}}))}(a)\\
&=P^{**}_{p_{\lambda}}\circ\pi^{**}_{\ell^{1}(S)}
\circ\rho_{\alpha}|_{\mathbb{M}_{E(\mathfrak{D}_{\lambda})}(\ell^{1}(G_{p_{\lambda}}))}(a)\rightarrow
a.
\end{split}
\end{equation}
Hence $
\mathbb{M}_{E(\mathfrak{D}_{\lambda})}(\ell^{1}(G_{p_{\lambda}}))$
is approximately biflat. By Theorem \ref{wit unit},
$\ell^{1}(G_{p_{\lambda}})$ is approximately biflat. Since
$\ell^{1}(G_{p_{\lambda}})$ is unital, by \cite[Theorem 2.4]{sam}
$\ell^{1}(G_{p_{\lambda}})$ is pseudo-amenable, hence  by
\cite[Proposition 4.1]{ghah pse} $G_{p_{\lambda}}$ is amenable.
Applying \cite[Theorem 3.7]{rost} to finish the proof.

Converse is clear.
\end{proof}

\begin{cor}
Let $S=\cup_{e\in E(S)}G_{e}$ be a Clifford semigroup such that
$E(S)$ is uniformly locally finite. If $\ell^{1}(S)^{**}$ is
approximately biflat, then $G_{e}$ is amenable for each $e\in E(S)$.
\end{cor}
\begin{proof}
Suppose  that $\ell^{1}(S)^{**}$ is approximately biflat. It is
well-known that $\ell^{1}(S)\cong\ell^{1}-\oplus_{e\in
E(S)}\ell^{1}(G_{e})$. Since  $\ell^{1}(G_{e})$ is unital, then
$\ell^{1}(S)$ has an approximate identity.  Applying  previous
Theorem, implies that $\ell^{1}(S)$ is pseudo-amenable. Then by
\cite[Theorem 3.7]{rost} $G_{e}$ is amenable, for each $e\in E(S).$
\end{proof}
\section{An application to Triangular Banach algebras}
In this section we give some examples of matrix algebras which is
never approximately biflat.

Let $A$ be a Banach algebra and $X$ be a Banach $A$-bimodule.
Suppose that $\phi\in\Delta(A).$ We say that $X$ has a left
$\phi$-character if there exists a non-zero map $\psi\in X^{*}$ such
that
$$\psi(a\cdot x)=\phi(a)\psi(x),\quad (a\in A,x\in X).$$ Similarly
we can define right case and two sided case.  It is easy to see that
for $\phi\in\Delta(A)$, $\phi\otimes\phi$ on $A\otimes_{p}A$ is a
left $\phi$-character. Also if $A$ has a closed ideal $I\subseteq
\ker\phi$, then $\overline{\phi}$ on $\frac{A}{I}$ is a left
$\phi$-character, where $\overline{\phi}:\frac{A}{L}\rightarrow
\mathbb{C}$ given by $\overline{\phi}(a+I)=\phi(a)$ for all $a\in
A.$

Let $A$ and $B$ be  Banach algebras and let $X$ be a Banach
$A,B$-module, that is, $X$ is a Banach space, a left $A$-module and
a right $B$-module which the compatible module action that satisfies
$(a\cdot x)\cdot b=a\cdot (x\cdot b)$ and $||a\cdot x\cdot b||\leq
||a||||x||||b||$ for every $a\in A, x\in X, b\in B$.  With the usual
matrix operation and $||\left(\begin{array}{cc} a&x\\
0&b\\
\end{array}
\right)||=||a||+||x||+||b||$,
$T=Tri(A,X,B)=\left(\begin{array}{cc} A&X\\
0&B\\
\end{array}
\right)$ becomes a Banach algebra which is called Triangular Banach
algebra.
Let $\phi\in\Delta(B)$. We define a character  $\psi_{\phi}\in\Delta(T)$  via  $\psi_{\phi}\left(\begin{array}{cc} a&x\\
0&b\\
\end{array}
\right)=\phi(b)$ for every $a\in A$, $b\in B$ and $x\in X$.
\begin{Theorem}\label{tri}
Let $T=Tri(A,X,B)$ with $\overline{A^{2}}=A$ and $\overline{A\cdot
X}=X$. Suppose that $\phi\in\Delta(B)$ with
$\overline{B\ker\phi}=\ker\phi$. If one of the followings hold
\begin{enumerate}
\item [(i)] $B$ is not left $\phi$-amenable;
\item [(ii)] $X$ has a right $\phi$-character;
\end{enumerate}
then $T^{**}$ is not approximately biflat.
\end{Theorem}
\begin{proof}
We go toward a contradiction and suppose that $T^{**}$ is
approximately biflat. Let $\psi_{\phi}$ be same as above. It is
clear that $\ker\psi_{\phi}=Tri(A,X,\ker\phi)$. Since
$\overline{A^{2}}=A$, $\overline{A\cdot X}=X$ and
$\overline{B\ker\phi}=\ker\phi$, then
$\overline{T\ker\psi_{\phi}}=\psi_{\phi}$. Then by Theorem \ref{for
tri}, $T$ is left $\psi_{\phi}$-amenable. Set $I=Tri(0,X,B)$. It is
clear that $I$ is
 closed ideal and $\psi_{\phi}|_{I}\neq 0$. Using \cite[Lemma 3.1]{kan}, one can see
that $I$ is  left $\psi_{\phi}$-amenable. Thus by \cite[Theorem
1.4]{kan}, there exists a bounded net $(i_{\alpha})$ in $I$ such
that
$$ii_{\alpha}-\psi_{\phi}(i)i_{\alpha}\rightarrow 0,\quad \psi_{\phi}(i_{\alpha})=1,\quad (i\in I).$$
Take $(x_{\alpha})$ in $X$ and $(b_{\alpha})$ in $B$ such that
$i_{\alpha}=\left(\begin{array}{cc} 0&x_{\alpha}\\
0&b_{\alpha}\\
\end{array}
\right)$. Hence we have $$\left(\begin{array}{cc} 0&x\\
0&b\\
\end{array}
\right)\left(\begin{array}{cc} 0&x_{\alpha}\\
0&b_{\alpha}\\
\end{array}
\right)-\psi_{\phi}(\left(\begin{array}{cc} 0&x\\
0&b\\
\end{array}
\right))\left(\begin{array}{cc} 0&x_{\alpha}\\
0&b_{\alpha}\\
\end{array}
\right)\rightarrow 0$$ and
$\quad \psi_{\phi}(\left(\begin{array}{cc} 0&x_{\alpha}\\
0&b_{\alpha}\\
\end{array}
\right))=\phi(b_{\alpha})=1$ for each $x\in X, b\in B.$ Thus we have
$$xb_{\alpha}-\phi(b)x_{\alpha}\rightarrow 0, bb_{\alpha}-\phi(b)b_{\alpha}\rightarrow 0,\phi(b_{\alpha})=1,\quad(x\in X,b\in B).$$
If $(i)$ holds the facts $$bb_{\alpha}-\phi(b)b_{\alpha}\rightarrow
0,\phi(b_{\alpha})=1,\quad(x\in X,b\in B)$$ give us a
contradiction(left $\phi$-amenability condition for $B$).

Suppose that $(ii)$ happens. Take $\eta $ as a right
$\phi$-character on $X$. Since
$xb_{\alpha}-\phi(b)x_{\alpha}\rightarrow 0, \phi(b_{\alpha})=1$ for
each $x\in X,b\in B$, then we have
$$\eta(xb_{\alpha}-\phi(b)x_{\alpha})=\eta(xb_{\alpha})-\phi(b)\eta(x_{\alpha})=\phi(b_{\alpha})\eta(x)-\phi(b)\eta(x_{\alpha})\rightarrow
0$$ for each $x\in X,b\in B.$ Thus we have $\lim
\phi(b)\eta(x_{\alpha})=\eta(x)$. Take $b\in\ker\phi$, then we have
$\eta(x)=0$ for each $x\in X$ which is a contradiction($\eta$ is a
non-zero functional).
\end{proof}
\begin{cor}
Let $G$ be a locally compact group. Then
$Tri(S(G),L^{1}(G),S(G))^{**}$ is not approximately biflat.
\end{cor}
\begin{proof}
Suppose that $\phi\in\Delta(S(G))$. It is well-known that $S(G)$ has
a left approximate identity and is a dense left ideal of $L^{1}(G)$
and also $L^{1}(G)$ has a bounded approximate identity. Then
$$\overline{S(G)^{2}}=S(G),\quad
\overline{S(G)L^{1}(G)}=L^{1}(G),\quad
\overline{S(G)\ker\phi}=\ker\phi.$$ Since $S(G)$ is a left ideal in
$L^{1}(G)$, by \cite[Lemma 2.2]{alagh} $\phi$ can be extended to a
character on $L^{1}(G)$,  which is a right $\phi$-character for
$L^{1}(G)$. Now apply  Theorem \ref{tri} to show that
$Tri(S(G),L^{1}(G),S(G))^{**}$ is not approximately biflat.
\end{proof}
\begin{cor}
Let $G$ be a locally compact group. Then
$Tri(L^{1}(G),S(G)\otimes_{p}S(G),S(G))^{**}$ is not approximately
biflat.
\end{cor}
\begin{proof}
Since $L^{1}(G)$ has a bounded approximate identity, we have
$\overline{L^{1}(G)^{2}}=L^{1}(G)$. Also using Cohn factorization
 theorem we have $\overline{L^{1}(G)\cdot (S(G)\otimes_{p}S(G))}=S(G)\otimes_{p}S(G)$.
 Since $S(G)$ has a left approximate identity, then
 $\overline{S(G)\ker\phi}=\ker\phi$ for each $\phi\in \Delta(S(G)).$
Note that for each $\phi \in \Delta(S(G))$, $\phi\otimes \phi$ which
is defined by $$\phi\otimes\phi(a\otimes b)=\phi(a)\phi(b)\quad
(a,b\in S(G))$$ is a right $\phi$-character on
$S(G)\otimes_{p}S(G)$. Then apply Theorem \ref{tri} to show that
$Tri(L^{1}(G),S(G)\otimes_{p}S(G),S(G))^{**}$ is not approximately
biflat.
\end{proof}
Similarly one can show the following result.
\begin{cor}
Let $G$ be a locally compact group. Then
$Tri(L^{1}(G),M(G),S(G))^{**}$ is not approximately biflat.
\end{cor}
\begin{small}

\end{small}
\end{document}